\documentclass[final,3p]{elsarticle}
 \usepackage{graphics}
 \usepackage{graphicx}
 \usepackage{epsfig}
\usepackage{amssymb}
 \usepackage{amsthm}
 \usepackage{lineno}
 \usepackage{amsmath}
   \numberwithin{equation}{section}
\usepackage{mathrsfs}

\newtheorem{thm}{Theorem}[section]

\newtheorem{con}[thm]{Conjecture}
\newtheorem{lem}[thm]{Lemma}

\newtheorem{rem}[thm]{Remark}

\biboptions{sort&compress,square}
\journal{ }
\begin{document}
\begin{frontmatter}
 \author{Jian Wang}
\author{Yong Wang\corref{cor2}}
\ead{wangy581@nenu.edu.cn}
\cortext[cor2]{Corresponding author. Tel.:+86 431 85099589; Fax.:+86 431 85098237.}
\address{School of Mathematics and Statistics, Northeast Normal University,
Changchun, 130024, P.R.China}
\title{Double Conformal Invariants and the Wodzicki Residue   }
\begin{abstract}
For compact real manifolds, a new double conformal invariant is constructed using the Wodzicki residue and the $d$ operator
in the framework of Connes. In the flat case, we compute this double conformal invariant, and in some special cases, we also
compute this double conformal invariants. For complex manifolds, a new double conformal invariant is constructed using
the Wodzicki residue and the $\bar{\partial}$ operator in the same way, and this double conformal invariant is computed
in the flat case.
\end{abstract}

\end{frontmatter}
\section{Introduction}
\label{1}
Wodzicki  discovered a {\rm trace} on the algebra of all classical pseudodifferential  operators on a closed compact manifold in \cite{Wo},
 which is called the Wodzicki residue now. In \cite{Co}, for an even dimensional compact oriented conformal real manifold without
 boundary, a canonical Fredholm module was constructed and a conformal invariant was defined  by the Wodzicki residue. Connes  also
 gave an explanation of the Polyakov action and its 4-dimensional analogy by using his conformal invariant.  In \cite{Ug1},
 Connes' result was generalized to the higher dimensional case and an explicit expression of  the Connes' invariant in the flat case
 was given by Ugalde. In addition, the way of computation in the general case was indicated.
 In \cite{Wa1} and \cite{Wa2}, Wang generalized the Connes' framework to the case of manifolds with
 boundary by using the noncommutative residue on Boutet de Monvel's algebra found in \cite{FGLS}. In \cite{Zu1} and \cite{Zu2},
 the Polyakov action on a compact Riemann
 surface was expressed by using $\bar{\partial}$ operator. Let $g_{1}$ and $g_{2}$ be two Riemannian matrics on manifolds.
 A invariant $I(g_{1}, ~g_{2})$ about these two metrics is called a double conformal invariant, if $I(f^{1}g_{1}, ~f^{2}g_{2})=I(g_{1},~~g_{2})$
 for positive smooth functions $f^{1}$ and $f^{2}$. The motivation of this paper is to find new double conformal invariants
 for real manifolds and complex manifolds by the Wodzicki residue.

 This paper is organized as follows: In Section 2, for a compact real manifold, following the Connes' framework, we construct
 a Fredholm module associated to the  $d$ operator and define a double conformal invariant. In Section 3, we compute the leading
 symbol of $F$ defined in Section 2, $\sigma(F^{g}_{d}) $, and  ${\rm trace}[\sigma_L(F^{g_{1}}_{d})(\xi)\sigma_L(F^{g_{2}}_{d})(\eta)] $
 for $\xi, \eta$ not zero in $T^{*}_{x}M$. In the flat case, we get the formula of
the double conformal invariants. In Section 4, we compute this double conformal invariant for 2-dimensional case.
 In Section 5, for a compact complex manifold, we construct a Fredholm module associated to the $\bar{\partial}$
 operator and define a double conformal invariant. In Section 6, we compute this double conformal invariant in the flat case. In section 7,
  we compute this double conformal invariant  for 1-dimensional case.

 \section{A Double Conformally Invariant Differential Form $\Omega^{g_{1},g_{2}}_{d}(f_1,f_2)$ for Compact Real Manifolds }
\label{2}
Let $M$ be a $n$ dimensional compact connect manifold. Let $g^{1}$ and $g^{2}$ be Riemannian metrics on the tangent bundle $TM$ of $M$,
and $\{e_{1},e_{2},\cdots,e_{n}\}$ be local orthonormal basis on $M$. Let $\{e^{1},e^{2},\cdots,e^{n} \}$ be the dual basis,
and $\Phi=e^{1}\wedge e^{2}\wedge \cdots \wedge e^{n}$ be the canonical volume of $M$. Define the metric on $p-$form space $A^{p}$
 \begin{equation}
  g(e^{i_{1}}\wedge e^{i_{2}}\wedge\cdots\wedge e^{i_{n}},~ e^{j_{1}}\wedge e^{j_{2}}\wedge\cdots\wedge e^{j_{n}})=det|\langle e^{i_{l}},e^{j_{k}}\rangle|.
\end{equation}
Define an inner product for $\psi,\eta\in A^{p}$

\begin{equation}
 \langle\langle\psi, \eta\rangle\rangle=\int_{M}\langle\psi, \eta\rangle\Phi,
\end{equation}
which makes $A^{p}$ be a preHilbert space and let $L^{2}(A^{p})$ be $L^{2}$-completion. Recall the Hodge $*$ operator
$*: A^{p}\rightarrow A^{n-p}$ defined as follows:

\begin{equation}
 \langle\psi, \eta\rangle\Phi=\psi \wedge *\eta   ~ {\rm for}~ \psi,\eta\in A^{p}.
\end{equation}
Let $\Delta$ be the $d$-Laplacian and $d^{*}=c_{0}*d*$ be the adjoint operator of  $d$~($d^{*}: A^{p}\rightarrow A^{p-1}$). Let
$H^{p}=ker\Delta_{p}$ be the harmonic $p$-forms space. Recall the Hodge Theorem 

\begin{equation}
 A^{p}=H^{p}\oplus Im d\oplus Im d^{*}.
\end{equation}
Now we take the conformal rescaling metric $\tilde{g}=e^{2f}g^{TM}$ for some smooth function $f$ on $M$.
Then $\{e^{f}e^{1},e^{f}e^{2},\cdots,e^{f}e^{n}\}$ are the associated dual orthonormal basis about $e^{2f}g^{TM}$.
Let $\hat{*}$ be the Hodge star operator associated to the rescaling metric. By \cite{Ug1}, we have $\hat{*}=e^{(2p-n)f}*$.
In particular, when $2p=n$, the $*$ operator is conformal invariant, that is $\hat{*}=*$.
By the definition of the inner product, when $2p=n$, $L^{2}(A^{p})$ is conformally invariant. In the following, we always
assume $2p=n$. By the Hodge decomposition theorem, we have $A^{\frac{n}{2}}=Im\triangle (A^{\frac{n}{2}})\oplus H^{\frac{n}{2}}$
and $H^{\frac{n}{2}}=ker(d)\cap ker(d*)$. $F_{d}$ is the pseudodifferential operator of order $0$ acting on $A^{\frac{n}{2}}$, obtained
from the orthogonal projection $P$ on the image of $d$, by the relation $F^{g}_{d}=2P_{Imd}-Id$. From the Hodge decomposition theorem, it is easy to
see that $F^{g}_{d}$ preserves the finite dimensional space of harmonic forms $H^{p}$, and the $F^{g}_{d}$ restricted to the $H\ominus H^{p}$ is given
by
\begin{equation}
F^{g}_{d}=\frac{d d^*-d^*d}{dd^*+d^*d},
\end{equation}
both $H^{p}$ and $F^{g}_{d}$ are independent of the metric in the conformal class of $g$. Then we have
$(C^{\infty}(M),~L^2(A^\frac{n}{2}),$ $F^{g}_{d})$ is a conformally invariant Fredholm module (for definition, see \cite{GVF}) .

Next we recall the definition of the Wodzicki residue. Let $\overline{M}$ be a $k$-dimensional Riemannian manifold
and $P$ is a pseudodifferential operator acting on sections of a vector bundle $B$ over the manifold $\overline{M}$.
In a local coordinate system, denote by $\sigma_{-k}(P)$ the $(-k)$-order symbol, then the Wodzicki residue of $P$
(for details, see \cite{Wo} or \cite{GVF}) is defined as follows:
\begin{equation}
{\rm Wres}(P)=\int_{\overline{M}}\int_{|\xi|=1}{\rm Tr}(\sigma_{-k}(P)(x,\xi))\sigma(\xi)dx,
\end{equation}
where $\xi$ is the cotangent vector and $\sigma(\xi)$ is the canonical $(k-1)$-sphere volume. ${\rm Wres}(P)$ is a
{\rm trace} which is independent of the choice of the local coordinates on $\overline{M}$ and the local basis of $B$.
In particular, the Wodzicki residue does not depend on the choice of the metric. Let $g_{1}^{TM},~ g_{2}^{TM}$ are
the two different metrics on $TM$. Let $F^{g_{i}}_{d}=2P_{\rm Im d}-{\rm Id}~(i=1, 2)$, where $P_{\rm Im d}$ is
the projection from $A^{p}$ to $ \rm Im d$. As in [Co], we have on $A^{p}\ominus H^{p}$, for $i=1,2$
\begin{equation}
F^{g_{i}}_{d}=\frac{d d_{g_{i}}^*-d_{g_{i}}^*d}{dd_{g_{i}}^*+d_{g_{i}}^*d}.
\end{equation}
Thus $F^{g_{i}}_{d}~(i=1,2)$ is a conformally invariant $0$-order pseudodifferential bounded operator. Define the $n$-form
$\Omega^{g_{1},g_{2}}_{d}(f_1,f_2)$ by the following identity for $f_0, f_1, f_2\in C^{\infty}( M)$
\begin{equation}
Wres(f_0[F^{g_{1}}_{d},f_1][F^{g_{2}}_{d},f_2])
=\int_Mf_0\Omega^{g_{1},g_{1}}_{d}(f_1,f_2).
\end{equation}
Then, as in \cite{Co} or \cite{Ug1}, we have
\begin{thm}\label{th:33}
$\Omega^{g_{1},g_{2}}_{d}(f_1,f_2)$ is a unique defined by (2.8), symmetric and double conformally invariant
$n$-differential form. Furthermore, $ \int_Mf_0\Omega^{g_{1},g_{2}}_{d}(f_1,f_2)$ defines a Hochschild $2$-cocycle
on the algebra of smooth functions on $M$. In particular, when $g_{1}=g_{2} $, we get the Connes' conformal invariant.
\end{thm}
\begin{proof}
Let \begin{equation}
\Psi(f_0, f_1, f_2)=Wres(f_0[F^{g_{1}}_{d},f_1][F^{g_{2}}_{d},f_2]),
\end{equation}
for $f_0, f_1, f_2, f_3\in C^{\infty}( M),$
by $[S, ~fh]=[S,~f]h+f[S,~h]$, then
\begin{eqnarray}
(b\Psi)(f_0, f_1, f_2, f_3)&=& Wres(f_0f_1[S_{1},f_2][S_{2},f_3])- Wres(f_0[S_{1},f_1f_2][S_{2},f_3]) \nonumber\\
&&+ Wres(f_0[S_{1},f_1][S_{2},f_2f_3])-Wres(f_3f_0[S_{1},f_1][S_{2},f_2])=0,
\end{eqnarray}
and $\Omega^{g_{1},g_{2}}_{d}(f_1,f_2)=\Omega^{g_{1},g_{2}}_{d}(f_2,f_1)$, we get $ \int_Mf_0\Omega^{g_{1},g_{2}}_{d}(f_1,f_2)$
is a Hochschild $2$-cocycle on the algebra of smooth functions on $M$.
\end{proof}

\section{The Computation of the Double Conformal Invariant $\Omega^{g_{1},g_{2}}_{d}(f_1,f_2)$  }
\label{3}
Next we compute the conformal invariant $\Omega^{g_{1},g_{2}}_{d}(f_1,f_2)$. Similar to Lemmas 2.2 and 2.3
in \cite{Ug1}, we have

\begin{equation}
\Omega^{g_{1},g_{2}}_{d}(f_1,f_2)=\int_{|\xi|=1}{\rm tr}_{A^{p}}
[\sum\frac{1}{\alpha'!\alpha''!\beta!\delta!}D^{\beta}_x({f_1}) D^{\alpha''+\delta}_x({f_2})
\partial^{\alpha'+\alpha''+\beta}_{\xi}(\sigma^{F^{g_{1}}_{d}}_{-j})
\partial^{\delta}_{\xi}D^{\alpha'}_x(\sigma^{F^{g_{2}}_{d}}_{-k})]\sigma(\xi)d^nx ,
\end{equation}
where $A^{p}$ is the $p$-form space and $\sigma_{-j}^{F^{g_{i}}_{d}}(i=1,2)$
denotes the  $-j$ order symbol of $F^{g_{i}}_{d}(i=1,2)$;
$D^\beta_x=(-i)^{|\beta|}\partial^\beta_x$ and the sum is taken over $|\alpha'|+|\alpha''|+|\beta|+|\delta|+j+k=n;
|\beta|\geq1,|\delta|\geq1; \alpha',\alpha'',\beta,\delta \in {\bf Z}^{n}_+; j,k\in {\bf Z}_+.$
Let $\xi=\sum_{i=1}^{n}\xi_je^j$ be the cotangent vector in $T^{*}M$. Let $\sigma_L(P)$ denote the leading
symbol of the pseudodifferential operator $P$ and $\varepsilon(.)$ and $\iota(.)$ denote the exterior and interior
multiplications respectively. By Lemma 3.52 in \cite{Gi}, we have

\begin{lem}\label{le:31}
When $g_{i}~(i=1,2)$ is flat, we have for $\xi\neq 0$
\begin{equation}
\sigma^{F}_{-k}(x, \xi)=0 ~ {\rm for}~\forall~ k>0;
\end{equation}
\begin{equation}
\sigma(F^{g_{i}}_{d})=\sigma_L(F^{L,g_{i}}_{d})
=\frac{\varepsilon(\xi)\iota^{g_{i}}(\xi)-\iota^{g_{i}}(\xi)\varepsilon(\xi)}{|\xi|_{g_{i}}^{2}} ~(i=1,2).
\end{equation}
\end{lem}

We assume that $g_{i}~(i=1,2)$ is flat. Similarly to Lemma 6.1 in \cite{Ug1}, then we have

\begin{lem}\label{le:32}
The following equality holds
\begin{equation}
\Omega^{g_{1},g_{2}}_{\overline{\partial},{\rm flat}}(f_1,f_2)=\int_{|\xi|=1}{\rm tr}_{A^{\frac{n}{2}}}
[\sum\frac{1}{\alpha!\beta!\delta!}D^{\beta}_x({f_1})D^{\alpha+\delta}_x({f_2})
\partial^{\alpha+\beta}_{\xi}(\sigma_L(F^{g_{1}}_{d}))
\partial^{\delta}_{\xi}(\sigma_L(F^{g_{2}}_{d}))]d^{n-1}\xi dx,
\end{equation}
where the sum is taken over $|\alpha|+|\beta|+|\delta|=n; ~1\leq |\beta|; ~1\leq |\delta|.$
\end{lem}
Let
\begin{equation}
\psi(\xi,\eta)={\rm Tr}_{A^{\frac{n}{2}}}[\sigma_L(F^{g_{1}}_{d})(x,\xi)\sigma_L(F^{g_{2}}_{d})(x,\eta)];
\end{equation}

\begin{equation}
T_{n}'\psi(\xi,\eta,u,v)=\sum\frac{u^{\beta}}{\beta!}\frac{v^{\delta}}{\delta!}{\rm tr}[\partial^\beta_\xi
(\sigma_L(F^{g_{1}}_{d})(\xi))\partial^\delta_\eta(\sigma_L(F^{g_{2}}_{d})(\eta))],
\end{equation}
where the sum is taken over $|\beta|+|\delta|=n,~|\beta|\geq 1,~|\delta|\geq 1$. $T_{n}'(\xi,\eta,u,v)$ is the
term of order $n$ in the Taylor expansion of $\psi(\xi+u,\eta+v)$ at $u=v=0$ minus the terms with only powers
 of $u$ or only powers of $v$. As in Section 4 in \cite{Ug1}, we have

\begin{thm}\label{th:33}
\begin{equation}
\Omega^{g_{1},g_{2}}_{d, {\rm flat}}(f_1,f_2)=(\sum A_{a, b}(D^a_xf_1)(D^b_xf_2))dx,
\end{equation}
where
\begin{equation}
\sum A_{a, b}u^av^b=\int_{|\xi|=1}(T_{n}'\psi(\xi,\xi,u+v,v)-T_{n}'\psi(\xi,\xi,v,v))\sigma(\xi).
\end{equation}
\end{thm}
By Theorem 3.3, to obtain an explicit expression for $\Omega^{g_{1},g_{2}}_{d, {\rm flat}}$, it is
necessary to study ${\rm Tr}_{A^{\frac{n}{2}}}[\sigma_L(F^{g_{1}}_{d})(\xi)$ $\sigma_L(F^{g_{2}}_{d})(\eta)]$
for $\xi$ and $\eta$ not zero in $T^{*}_xM$.

\begin{thm}\label{th:34}
With $\sigma_L(F^{g_{1}}_{d})(\xi)\sigma_L(F^{g_{2}}_{d})(\eta)$ acting on m-forms, we have

For $m=1$
\begin{equation}
{\rm Tr}_{A^{1}}[\sigma_L(F^{g_{1}}_{d})(\xi)\sigma_L(F^{g_{2}}_{d})(\eta)]
=4\frac{\langle\xi, \eta\rangle_{g_{1}}\langle\xi, \eta\rangle_{g_{2}}}{|\xi|_{g_{1}}^2|\eta|_{g_{2}}^2}-3C^0_n.
\end{equation}

For $m\geq 2$
\begin{equation}
{\rm Tr}_{A^{m}}[\sigma_L(F^{g_{1}}_{d})(\xi)\sigma_L(F^{g_{2}}_{d})(\eta)]
=a_{n,m}\frac{\langle\xi, \eta\rangle_{g_{1}}\langle\xi, \eta\rangle_{g_{2}}}{|\xi|^2|\eta|^2}+b_{n,m},
\end{equation}
where $\langle\xi, \eta\rangle_{g_{i}}(i=1,2)$ represents the inner product $g_{i}(\xi, \eta)~(i=1,2)$
and

\begin{equation}
b_{n,m}=C_{n-2}^{m-2}+C_{n-2}^{m}-2C_{n-2}^{m-1}; ~a_{n,m}=C_n^m-C_{n-2}^{m-2}-C_{n-2}^{m}+2C_{n-2}^{m-1},
\end{equation}
where $C_n^m$ is a combinatorial number.
\end{thm}

\begin{lem}\label{le:35}
With $\sigma_L(F^{g_{1}}_{d})(\xi)\sigma_L(F^{g_{2}}_{d})(\eta)$ acting on m-forms, we have
\begin{equation}
\rm tr(\iota^{g_{1}}_{m+1}(\eta)\varepsilon_{m}(\xi))=\langle\xi,\eta\rangle_{g_{1}}A_{n,m},
\end{equation}
where $A_{n,m}=C_n^m-C_n^{m-1}+\cdots +(-1)^m C_n^0$ and $\varepsilon_{m} $ denotes the operator $ \varepsilon $ acting
on the space of $m$ forms.
\end{lem}

\begin{proof}
By using the {\rm trace} property and the relation
\begin{equation}
\varepsilon_{m-1}(\xi)\iota^{g_{1}}_{m}(\eta)+\iota^{g_{1}}_{m+1}(\eta)\varepsilon_{m}
(\xi)=\langle\xi,\eta\rangle_{g_{1}}I_{m},
\end{equation}
where $I_{m}$ is the identity on $A^{m}$-forms, we get
\begin{equation}
{\rm tr}[\iota^{g_{1}}_{m}(\eta)\varepsilon_{m-1}(\xi)]+{\rm tr}[\iota^{g_{1}}_{m+1}(\eta)\varepsilon_{m}(\xi)]=\langle\xi,\eta\rangle_{g_{1}}C_n^m,
\end{equation}
when $m=0$, we have
${\rm tr}[\iota^{g_{1}}_{1}(\eta)\varepsilon_{0}(\xi)]=\langle\xi,\eta\rangle_{g_{1}}I_{0}$.
Then by (3.14), we can prove this lemma.
\end{proof}

\noindent{\bf Proof of Theorem 3.4}
\begin{proof}
 Using  Lemma 3.5 and (3.13) we deduce
\begin{eqnarray}
&&{\rm trace}_{A^{m}}[\sigma_L(F^{g_{1}}_{d})(x, \xi)\sigma_L(F^{g_{2}}_{d})(x, \eta)]\nonumber\\
&=&|\xi|_{g_{1}}^{-2}|\eta|_{g_{2}}^{-2} {\rm trace}_{A^{m}}[(2\varepsilon_{m-1}(\xi)\iota^{g_{1}}_{m}(\xi)-|\xi|_{g_{1}}^{-2}I_{m})
(2\varepsilon_{m-1}(\eta)\iota^{g_{2}}_{m}(\eta)-|\eta|_{g_{2}}^{-2}I_{m})] \nonumber\\
&=&4|\xi|_{g_{1}}^{-2}|\eta|_{g_{2}}^{-2}{\rm trace}_{A^{m}}(\varepsilon_{m-1}(\xi)\iota^{g_{1}}_{m}(\xi)\varepsilon_{m-1}(\eta)\iota^{g_{2}}_{m}(\eta))
-4A_{n,m-1}+C_{n}^{m},
\end{eqnarray}
when $m=1$ and by Lemma 3.5, we can prove (3.9).

For $m\geq 1$ and $\xi_1,\xi_2,\eta_1,\eta_1\in A_{1}$, let

\begin{equation}
a_{m}(\xi_1,\xi_2,\eta_1,\eta_2)={\rm trace}_{A^{m}}[\varepsilon_{m-1}(\xi_1)\iota^{g_{1}}_{m}(\xi_2)
\varepsilon_{m-1}(\eta_1)\iota^{g_{2}}_{m}(\eta_2)],
\end{equation}
we have the relation
\begin{eqnarray}
&&a_{m}(\xi_1,\xi_2,\eta_1,\eta_2) \nonumber\\
&=&{\rm trace}_{A^{m}}\{[\langle\xi_1,\xi_2\rangle_{g_{1}}I_{m}-\iota^{g_{1}}_{m+1}(\xi_2)\varepsilon_{m}(\xi_1)]
[\langle\eta_1,\eta_2\rangle_{g_{2}}I_{m}-\iota^{g_{2}}_{m+1}(\eta_2)\varepsilon_{m}(\eta_1)]\} \nonumber\\
&=&\langle\xi_1,\xi_2\rangle_{g_{1}}\langle\eta_1,\eta_2\rangle_{g_{2}}(C_n^m-2A_{n,m})+{\rm trace}_{A^{m}}[\varepsilon_{m}(\eta_1)\iota^{g_{1}}_{m+1}(\xi_2)
\varepsilon_{m}(\xi_1)\iota^{g_{2}}_{m+1}(\eta_2)],
\end{eqnarray}
which implies
\begin{equation}
a_{m+1}(\eta_1,\xi_2,\xi_1,\eta_2)=a_{m}(\xi_1,\xi_2,\eta_1,\eta_2)+\langle\xi_1,\xi_2\rangle_{g_{1}}\langle\eta_1,\eta_2\rangle_{g_{2}}(2A_{n,m}-C_n^m),
\end{equation}
and
\begin{equation}
a_{1}(\xi_1,\xi_2,\eta_1,\eta_2)=\langle\xi_1,\eta_2\rangle_{g_{2}}\langle\xi_2,\eta_1\rangle_{g_{1}}.
\end{equation}

Thus because of the relation
\begin{equation}
{\rm trace}_{A^{m}}[\sigma_L(F^{g_{1}}_{d})(x, \xi)\sigma_L(F^{g_{2}}_{d})(x, \eta)]
=4a_{m}(\xi,\xi,\eta,\eta)|\xi|_{g_{1}}^{-2}|\eta|_{g_{2}}^{-2}-4A_{n,m-1}+C_{n}^{m},
\end{equation}
we have a recursive way of computing the left hand side. This is enough to prove that

\begin{equation}
{\rm Tr}_{A^{m}}[\sigma_L(F^{g_{1}}_{d})(\xi)\sigma_L(F^{g_{2}}_{d})(\eta)]
=a_{n,m}\frac{\langle\xi, \eta\rangle_{g_{1}}\langle\xi, \eta\rangle_{g_{2}}}{|\xi|_{g_{1}}^2|\eta|_{g_{2}}^2}+b_{n,m},
\end{equation}
for some constants $a_{n,m}$ and $b_{n,m}$.

Now in the particular case in which $\xi=\eta=e^{j} $ is a member of an orthonormal  basis of $1-$ forms about $g_{2}$.
We compute $a_{n,m}$ and $b_{n,m}$ by (3.21) and taking special $\xi, \eta$. Then we have

\begin{equation}
\sigma_L(F^{g_{1}}_{d})(\xi)\sigma_L(F^{g_{2}}_{d})(\eta)
=\frac{\varepsilon(\xi)\iota^{g_{1}}(\xi)-\iota^{g_{1}}(\xi)\varepsilon(\xi)}{|\xi|_{g_{1}}^{2}}
\times[\varepsilon(\eta)\iota^{g_{2}}(\eta)-\iota^{g_{2}}(\eta)\varepsilon(\eta)].
\end{equation}

For a basic $m-$ form $\{e^{j_{1}}\wedge e^{j_{2}}\wedge \cdots \wedge e^{j_{m}}\}$ , one has

\begin{equation}
(\varepsilon_{m-1}(\xi)\iota^{g_{2}}_{m}(\xi)-\iota^{g_{2}}_{m+1}(\xi)\varepsilon_{m}(\xi))
(e^{j_{1}}\wedge e^{j_{2}}\wedge \cdots \wedge e^{j_{m}})=e^{j_{1}}\wedge e^{j_{2}}\wedge \cdots \wedge e^{j_{m}},
\end{equation}
if $ \xi=e^{j_{i}}$, for some $i$; and

\begin{equation}
(\varepsilon_{m-1}(\xi)\iota^{g_{2}}_{m}(\xi)-\iota^{g_{2}}_{m+1}(\xi)\varepsilon_{m}(\xi))
(e^{j_{1}}\wedge e^{j_{2}}\wedge \cdots \wedge e^{j_{m}})=-e^{j_{1}}\wedge e^{j_{2}}\wedge \cdots \wedge e^{j_{m}},
\end{equation}
if $ \xi\neq e^{j_{i}}$, for every $i$.

\begin{equation}
(\varepsilon_{m-1}(\xi)\iota^{g_{1}}_{m}(\xi)-\iota^{g_{1}}_{m+1}(\xi)\varepsilon_{m}(\xi))
(e^{j_{1}}\wedge e^{j_{2}}\wedge \cdots \wedge e^{j_{m}})
=|\xi|_{g_{1}}^{2}(e^{j_{1}}\wedge e^{j_{2}}\wedge \cdots \wedge e^{j_{m}}),
\end{equation}
 if $ \xi=e^{j_{i}}$, for some $i$; and
\begin{eqnarray}
&&(\varepsilon_{m-1}(\xi)\iota^{g_{1}}_{m}(\xi)-\iota^{g_{1}}_{m+1}(\xi)\varepsilon_{m}(\xi))
(e^{j_{1}}\wedge e^{j_{2}}\wedge \cdots \wedge e^{j_{m}}) \nonumber\\
&=&2\sum_{k=1}^{m} g_{1}(\xi,e^{j_{k}})(-1)^{k-1}(\xi\wedge e^{j_{1}}\wedge e^{j_{2}}\wedge \cdots \wedge \widehat e^{j_{k}}\cdots \wedge e^{j_{m}})
-|\xi|_{g_{1}}^{2}(e^{j_{1}}\wedge e^{j_{2}}\wedge \cdots \wedge e^{j_{m}}),
\end{eqnarray}
if $ \xi\neq e^{j_{i}}$, for every $i$.

For $\xi=\eta=e^{j_{i}}$,
\begin{equation}
\sigma_L(F^{g_{1}}_{d})(\xi)\sigma_L(F^{g_{2}}_{d})(\eta)(e^{j_{1}}\wedge e^{j_{2}}\wedge \cdots \wedge e^{j_{m}})
=e^{j_{1}}\wedge e^{j_{2}}\wedge \cdots \wedge e^{j_{m}}.
\end{equation}
 If both $\xi$, $\eta $ do not belong to $\{e^{j_{1}}\wedge e^{j_{2}}\wedge \cdots \wedge e^{j_{m}}\}$, then
\begin{eqnarray}
&&\sigma_L(F^{g_{1}}_{d})(\xi)\sigma_L(F^{g_{2}}_{d})(\eta)(e^{j_{1}}\wedge e^{j_{2}}\wedge \cdots \wedge e^{j_{m}}) \nonumber\\
&=&2|\xi|_{g_{1}}^{-2}\sum_{k=1}^{m} g_{1}(\xi,e^{j_{k}})(-1)^k(\xi\wedge e^{j_{1}}\wedge e^{j_{2}}\wedge \cdots \wedge \widehat e^{j_{k}}\cdots \wedge e^{j_{m}})
+e^{j_{1}}\wedge e^{j_{2}}\wedge \cdots \wedge e^{j_{m}},
\end{eqnarray}
since the first term of (3.28) does not use for computing the {\rm trace}. For $\xi=\eta=e^{j_{1}}$, we get
\begin{equation}
{\rm Tr}_{A^{m}}[\sigma_L(F^{g_{1}}_{d})(\xi)\sigma_L(F^{g_{2}}_{d})(\eta)]
=C_{n}^{m},
\end{equation}
therefore, $a_{n,m}+b_{n,m}=C_{n}^{m}$.

If both $\xi$ and $\eta $ are different members of an orthonomal basis of $g_{2}$, the term $g_{2}(\xi, \eta)=0$ and
the expression (3.10) reduce to $b_{n,m}$.

For $\xi,~\eta \in \{e^{j_{i}}, \cdots,~ e^{j_{m}}\}$, by (3.23) and (3.25), then
\begin{eqnarray}
&&\sigma_L(F^{g_{1}}_{d})(\xi)\sigma_L(F^{g_{2}}_{d})(\eta)(e^{j_{1}}\wedge e^{j_{2}}\wedge \cdots \wedge e^{j_{m}})\nonumber\\
&=&\frac{\varepsilon(\xi)\iota^{g_{1}}(\xi)-\iota^{g_{1}}(\xi)\varepsilon(\xi)}{|\xi|_{g_{1}}^{2}}
[\varepsilon(\xi)\iota^{g_{2}}(\xi)-\iota^{g_{2}}(\xi)\varepsilon(\xi)]e^{j_{1}}\wedge e^{j_{2}}\wedge \cdots \wedge e^{j_{m}}\nonumber\\
&=&\frac{\varepsilon(\xi)\iota^{g_{1}}(\xi)-\iota^{g_{1}}(\xi)\varepsilon(\xi)}{|\xi|_{g_{1}}^{2}}
e^{j_{1}}\wedge e^{j_{2}}\wedge \cdots \wedge e^{j_{m}}=e^{j_{1}}\wedge e^{j_{2}}\wedge \cdots \wedge e^{j_{m}};
\end{eqnarray}
and by (3.24) and (3.26), then
\begin{eqnarray}
&&\sigma_L(F^{g_{1}}_{d})(\xi)\sigma_L(F^{g_{2}}_{d})(\eta)(e^{j_{1}}\wedge e^{j_{2}}\wedge \cdots \wedge e^{j_{m}})
=\sigma_L(F^{g_{1}}_{d})(\xi)[-(e^{j_{1}}\wedge e^{j_{2}}\wedge \cdots \wedge e^{j_{m}})] \nonumber\\
&=&2|\xi|_{g_{1}}^{-2}\sum_{k=1}^{m} g_{1}(\xi,e_{j_{k}})(-1)^{k}(\xi\wedge e^{j_{1}}\wedge e^{j_{2}}\wedge
\cdots \wedge \widehat e^{j_{k}}\cdots \wedge e^{j_{m}})
+(e^{j_{1}}\wedge e^{j_{2}}\wedge \cdots \wedge e^{j_{m}}),
\end{eqnarray}
if both $\xi$, $\eta $ do not belong to $\{e^{j_{i}}, \cdots,~ e^{j_{m}}\}$; and
\begin{eqnarray}
\sigma_L(F^{g_{1}}_{d})(\xi)\sigma_L(F^{g_{2}}_{d})(\eta)(e^{j_{1}}\wedge e^{j_{2}}\wedge \cdots \wedge e^{j_{m}})
&=&\sigma_L(F^{g_{1}}_{d})(\xi)[-(e^{j_{1}}\wedge e^{j_{2}}\wedge \cdots \wedge e^{j_{m}})] \nonumber\\
&=&-e^{j_{1}}\wedge e^{j_{2}}\wedge \cdots \wedge e^{j_{m}},
\end{eqnarray}
if only $\xi$ belongs to $\{e^{j_{i}}, \cdots,~ e^{j_{m}}\}$ and  $\eta $ does not; and
\begin{eqnarray}
&&\sigma_L(F^{g_{1}}_{d})(\xi)\sigma_L(F^{g_{2}}_{d})(\eta)(e^{j_{1}}\wedge e^{j_{2}}\wedge \cdots \wedge e^{j_{m}})
=\sigma_L(F^{g_{1}}_{d})(\xi)(e^{j_{1}}\wedge e^{j_{2}}\wedge \cdots \wedge e^{j_{m}}) \nonumber\\
&=&2|\xi|_{g_{1}}^{-2}\sum_{k=1}^{m} g_{1}(\xi,e^{j_{k}})(-1)^{k-1}(\xi\wedge e^{j_{1}}\wedge e^{j_{2}}\wedge \cdots
\wedge \widehat e^{j_{k}}\cdots \wedge e^{j_{m}})
-(e^{j_{1}}\wedge e^{j_{2}}\wedge \cdots \wedge e^{j_{m}}),
\end{eqnarray}
if only $\eta $ belongs to $\{e^{j_{i}}, \cdots,~ e^{j_{m}}\}$ and $\xi$ does not, the first term of (3.31) and (3.33) 
does not use for computing the {\rm trace}.

 So the number of basic $m-$ forms  containing both $\xi$ and $\eta $  is $C_{n-2}^{m-2}$, and the number of
basic $m-$ forms containing neither $\xi$ nor $\eta $ is $C_{n-2}^{m}$, the number of basic $m-$ forms
containing exactly one of $\xi$ or $\eta $ is $2C_{n-2}^{m-1}$. So the value of $b_{n,m}$ is given by the
{\rm trace} of the above operator, which equals $b_{n,m}=C_{n-2}^{m-2}+C_{n-2}^{m}-2C_{n-2}^{m-1}$, hence
$a_{n,m}=C_n^m-C_{n-2}^{m-2}-C_{n-2}^{m}+2C_{n-2}^{m-1}$, and the proof is complete.
\end{proof}
\section{$\Omega^{g_{1},g_{2}}_{d}(f_1,f_2)$ for 2-dimensional Compact Real Manifolds }
\label{4}
In this section, we compute this conformal invariant for 2-dimensional compact manifolds. Let
$g^{1}$ and $g^{2}$  are  Riemannian metrics on the real tangent bundle $TM$ of $M$.
For $n=2$, let $M=S^{1}\times S^{1}$ and $g^{1}=g^{S^{1}}\oplus g^{S^{1}} $
be the orthonormal metric and $g^{2}$ be a different metric defined by
\begin{equation}
 g^{2}(e^{1},e^{1})=f;~g^{2}(e^{1},e^{2})=0; ~g^{2}(e^{2},e^{2})=h;~~(f>h>0)
 \end{equation}
where $e^{1},e^{2}$ are global othonormal basis of $g^{1}$.

By lemma 3.2 (for the definition of $\Omega^{g_{1},g_{2}}_{d}(f_1,f_2)$), since the sum is taken
$|\alpha|+|\beta|+|\delta|=2; ~1\leq |\beta|; ~1\leq |\delta|$, we get $|\alpha|=0,~|\beta|=|\delta|=1$, then we have the four cases:
$\{\beta =(1,0),~ \delta=(1,0)\};~\{\beta =(1,0),~\delta =(0,1)\};~\{\beta =(0,1),~\delta =(1,0)\};~\{\beta =(0,1),~\delta =(0,1)\}.$

Then we have
\begin{eqnarray}
\Omega^{g_{1},g_{2}}_{d}(f_1,f_2)
&=&\int_{|\xi|=1}{\rm tr}_{A^{1}}
[\sum\frac{1}{\alpha!\beta!\delta!}D^{\beta}_x({f_1})D^{\alpha+\delta}_x({f_2})\partial^{\alpha+\beta}_{\xi}(\sigma_L(F^{g_{2}}_{\overline{\partial}}))
\partial^{\delta}_{\xi}(\sigma_L(F^{g_{2}}_{\overline{\partial}}))]d\xi dx \nonumber\\
&=&\int_{|\xi|=1}\frac{1}{1!0!1!0!} D_{x}^{(1,0)}(f_{1}) D_{x}^{(1,0)}(f_{2})\partial_{(\xi_{1},~\xi_{2})}^{(1,0)}\partial_{(\eta_{1},~
\eta_{2})}^{(1,0)} {\rm trace}(\sigma_{L}^{F_{g_{1}}}(\xi) \sigma_{L}^{F_{g_{2}}}(\eta)) d\xi dx \nonumber\\
&&+\int_{|\xi|=1}\frac{1}{1!0!0!1!} D_{x}^{(1,0)}(f_{1}) D_{x}^{(0,1)}(f_{2})\partial_{(\xi_{1},~\xi_{2})}^{(1,0)}\partial_{(\eta_{1},~
\eta_{2})}^{(0,1)} {\rm trace}(\sigma_{L}^{F_{g_{1}}}(\xi) \sigma_{L}^{F_{g_{2}}}(\eta)) d\xi dx \nonumber\\
&&+\int_{|\xi|=1}\frac{1}{0!1!1!0!} D_{x}^{(0,1)}(f_{1}) D_{x}^{(1,0)}(f_{2})\partial_{(\xi_{1},~\xi_{2})}^{(0,1)}\partial_{(\eta_{1},~
\eta_{2})}^{(1,0)} {\rm trace}(\sigma_{L}^{F_{g_{1}}}(\xi) \sigma_{L}^{F_{g_{2}}}(\eta)) d\xi dx \nonumber\\
&&+\int_{|\xi|=1}\frac{1}{0!1!0!1!} D_{x}^{(0,1)}(f_{1}) D_{x}^{(0,1)}(f_{2})\partial_{(\xi_{1},~\xi_{2})}^{(0,1)}\partial_{(\eta_{1},~
\eta_{2})}^{(0,1)} {\rm trace}(\sigma_{L}^{F_{g_{1}}}(\xi) \sigma_{L}^{F_{g_{2}}}(\eta)) d\xi dx.
\end{eqnarray}

Since $n=2$, we get
\begin{eqnarray}
\Omega^{g_{1},g_{2}}_{d}(f_1,f_2)
&=&D_{x}^{(1,0)}(f_{1}) D_{x}^{(1,0)}(f_{2})\int_{|\xi|=1}\partial_{\xi_{1}}\partial_{\eta_{1}}
 {\rm trace}(\sigma_{L}^{F_{g_{1}}}(\xi) \sigma_{L}^{F_{g_{2}}}(\eta)) d\xi dx \nonumber\\
&&+D_{x}^{(1,0)}(f_{1}) D_{x}^{(0,1)}(f_{2})\int_{|\xi|=1}\partial_{\xi_{1}}\partial_{
\eta_{2}} {\rm trace}(\sigma_{L}^{F_{g_{1}}}(\xi) \sigma_{L}^{F_{g_{2}}}(\eta)) d\xi dx \nonumber\\
&&+D_{x}^{(0,1)}(f_{1}) D_{x}^{(1,0)}(f_{2})\int_{|\xi|=1}\partial_{\xi_{2}}\partial_{\eta_{1}
} {\rm trace}(\sigma_{L}^{F_{g_{1}}}(\xi) \sigma_{L}^{F_{g_{2}}}(\eta)) d\xi dx \nonumber\\
&&+D_{x}^{(0,1)}(f_{1}) D_{x}^{(0,1)}(f_{2})\int_{|\xi|=1}\partial_{\xi_{2}}\partial_{
\eta_{2}}{\rm trace}(\sigma_{L}^{F_{g_{1}}}(\xi) \sigma_{L}^{F_{g_{2}}}(\eta)) d\xi dx\nonumber\\
&=&D_{x}^{(1,0)}(f_{1}) D_{x}^{(1,0)}(f_{2})\int_{|\xi|=1}A_{1}d\xi dx+D_{x}^{(1,0)}(f_{1}) D_{x}^{(0,1)}(f_{2})\int_{|\xi|=1}A_{2}d\xi dx\nonumber\\
&&+D_{x}^{(0,1)}(f_{1}) D_{x}^{(1,0)}(f_{2})\int_{|\xi|=1}A_{3}d\xi dx+D_{x}^{(0,1)}(f_{1}) D_{x}^{(0,1)}(f_{2})\int_{|\xi|=1}A_{4}d\xi dx.
\end{eqnarray}

In this subsection we denote $\mid_{|\xi|=1}^{\xi=\eta}$ by $ |_{*}$. For $\xi=\sum_{1\leq i \leq n}\xi_ie^i$,
using the Theorem 3.4 and $n=2$, we have

\begin{equation}
 {\rm trace}(\sigma_{L}^{F_{g_{1}}}(\xi) \sigma_{L}^{F_{g_{2}}}(\eta))
 =4\frac{(\xi_{1}\eta_{1}+\xi_{2}\eta_{2})(f\xi_{1}\eta_{1}+h\xi_{2}\eta_{2})}{|\xi|_{g_{1}}^2|\eta|_{g_{2}}^2}-3.
\end{equation}

Through the computation
\begin{equation}
A_{1}=\partial_{\xi_{1}}\partial_{\eta_{1}} {\rm trace}(\sigma_{L}^{F_{g_{1}}}(\xi) \sigma_{L}^{F_{g_{2}}}(\eta))|_{*}
=1-2\xi_{1}^{2}+\frac{f(1+2\xi_{1}^{2})}{f\xi_{1}^{2}+h\xi_{2}^{2}}-\frac{2f^{2}\xi_{1}^{2}}{(f\xi_{1}^{2}+h\xi_{2}^{2})^{2}};
\end{equation}
\begin{equation}
A_{2}=\partial_{\xi_{1}}\partial_{\eta_{2}} {\rm trace}(\sigma_{L}^{F_{g_{1}}}(\xi) \sigma_{L}^{F_{g_{2}}}(\eta))|_{*}
=-2\xi_{1}\xi_{2}+\frac{(f+h)\xi_{1}\xi_{2}}{f\xi_{1}^{2}+h\xi_{2}^{2}}-\frac{2fh\xi_{1}\xi_{2}}{(f\xi_{1}^{2}+h\xi_{2}^{2})^{2}};
\end{equation}
\begin{equation}
A_{3}=\partial_{\xi_{2}}\partial_{\eta_{1}} {\rm trace}(\sigma_{L}^{F_{g_{1}}}(\xi) \sigma_{L}^{F_{g_{2}}}(\eta))|_{*}
=-2\xi_{1}\xi_{2}+\frac{(f+h)\xi_{1}\xi_{2}}{f\xi_{1}^{2}+h\xi_{2}^{2}}-\frac{2fh\xi_{1}\xi_{2}}{(f\xi_{1}^{2}+h\xi_{2}^{2})^{2}};
\end{equation}
\begin{equation}
A_{4}=\partial_{\xi_{2}}\partial_{\eta_{2}}{\rm trace}(\sigma_{L}^{F_{g_{1}}}(\xi) \sigma_{L}^{F_{g_{2}}}(\eta))|_{*}
=1-2\xi_{2}^{2}+\frac{h(1+2\xi_{2}^{2})}{f\xi_{1}^{2}+h\xi_{2}^{2}}-\frac{2h^{2}\xi_{2}^{2}}{(f\xi_{1}^{2}+h\xi_{2}^{2})^{2}}.
\end{equation}

Since $|\xi|=1$ is the unit circle, let $\xi_{1}=cos\theta, ~ \xi_{2}=sin\theta $. By (4.3)and(4.5), we have

\begin{equation}
\int_{|\xi|=1}A_{1}d\xi=\int_{0}^{2\pi}(1-2cos^{2}\theta)d\theta
+f\int_{0}^{2\pi}\frac{1+2cos^{2}\theta}{f cos^{2}\theta+h sin^{2}\theta} d\theta
-2f^{2}\int_{0}^{2\pi} \frac{cos^{2}\theta}{(fcos^{2}\theta+h sin^{2}\theta)^{2}} d\theta.
\end{equation}

By the integral formula and $f>h>0$, we have

\begin{equation}
\int_{0}^{2\pi}(1-2cos^{2}\theta)d\theta=0;
~\int_{0}^{2\pi}\frac{cos^{2}\theta}{f cos^{2}\theta+h sin^{2}\theta} d\theta=\frac{2\pi}{f-h}-\frac{2\pi h}{\sqrt{fh}(f-h)};
\end{equation}
\begin{equation}
\int_{0}^{2\pi}\frac{1}{f cos^{2}\theta+h sin^{2}\theta} d\theta=\frac{2\pi}{\sqrt{fh}},
~\int_{0}^{2\pi} \frac{cos^{2}\theta}{(fcos^{2}\theta+h sin^{2}\theta)^{2}} d\theta=\frac{2\pi}{f\sqrt{fh}}.
\end{equation}

By (4.9), (4.10) and (4.11), we get $\int_{|\xi|=1}A_{1}d\xi=\frac{4\pi\sqrt{f}}{\sqrt{f}+\sqrt{h}} $.

Similar to the computation of the first case, by integral formula, then
$\int_{|\xi|=1}A_{2}d\xi=\int_{|\xi|=1}A_{3}d\xi=0$, $\int_{|\xi|=1}A_{4}d\xi=\frac{4\pi\sqrt{h}}{\sqrt{f}+\sqrt{h}}$.

Therefore, for $n=2$, we compute the double conformally invariant

\begin{equation}
\Omega^{g_{1},g_{2}}_{d}(f_1,f_2)
=\frac{4\pi}{\sqrt{f}+\sqrt{h}}(D_{x}^{(1,0)}(f_{1}) D_{x}^{(1,0)}(f_{2})\sqrt{f}dvol_{g^{S^{1}}\oplus g^{S^{1}}}
+D_{x}^{(0,1)}(f_{1}) D_{x}^{(0,1)}(f_{2})\sqrt{h}dvol_{g^{S^{1}}\oplus g^{S^{1}}}).
\end{equation}

Next, for $f>h>0$, we let
\begin{equation}
  g^{2}(e^{1},e^{1})=f; ~g^{2}(e^{1},e^{2})=h; ~g^{2}(e^{2},e^{2})=f.
\end{equation}

Similar to the computation of 4.2, for $n=2$, we have
\begin{eqnarray}
\Omega^{g_{1},g_{2}}_{d}(f_1,f_2)
&=&D_{x}^{(1,0)}(f_{1}) D_{x}^{(1,0)}(f_{2})\int_{|\xi|=1}B_{1}d\xi dx+D_{x}^{(1,0)}(f_{1}) D_{x}^{(0,1)}(f_{2})\int_{|\xi|=1}B_{2}d\xi dx\nonumber\\
&&+D_{x}^{(0,1)}(f_{1}) D_{x}^{(1,0)}(f_{2})\int_{|\xi|=1}B_{3}d\xi dx+D_{x}^{(0,1)}(f_{1}) D_{x}^{(0,1)}(f_{2})\int_{|\xi|=1}B_{4}d\xi dx.
\end{eqnarray}

For $\xi=\sum_{1\leq i \leq n}\xi_ie^i$, using the Theorem 3.4, we have
\begin{equation}
 {\rm trace}(\sigma_{L}^{F_{g_{1}}}(\xi) \sigma_{L}^{F_{g_{2}}}(\eta))
 =4\frac{(\xi_{1}\eta_{1}+\xi_{2}\eta_{2})(f\xi_{1}\eta_{1}+f\xi_{2}\eta_{2}+h\xi_{1}\eta_{2}+h\xi_{2}\eta_{1})}{|\xi|_{g_{1}}^2|\eta|_{g_{2}}^2}-3.
\end{equation}

Through the computation
\begin{equation}
B_{1}=\partial_{\xi_{1}}\partial_{\eta_{1}} {\rm trace}(\sigma_{L}^{F_{g_{1}}}(\xi) \sigma_{L}^{F_{g_{2}}}(\eta))|_{*}
=1-2\xi_{1}^{2}+\frac{f(1+2\xi_{1}^{2})+2h\xi_{1}\xi_{2}}{f+2h\xi_{1}\xi_{2}}
-\frac{2(f\xi_{1}+h\xi_{2})^{2}}{(f+2h\xi_{1}\xi_{2})^{2}};
\end{equation}
\begin{equation}
B_{2}=\partial_{\xi_{1}}\partial_{\eta_{2}} {\rm trace}(\sigma_{L}^{F_{g_{1}}}(\xi) \sigma_{L}^{F_{g_{2}}}(\eta))|_{*}
=-2\xi_{1}\xi_{2}+\frac{2f\xi_{1}\xi_{2}+2h}{f+2h\xi_{1}\xi_{2}}
-\frac{2[fh+(f^{2}+h^{2})\xi_{1}\xi_{2}]^{2}}{(f+2h\xi_{1}\xi_{2})^{2}};
\end{equation}
\begin{equation}
B_{3}=\partial_{\xi_{2}}\partial_{\eta_{1}} {\rm trace}(\sigma_{L}^{F_{g_{1}}}(\xi) \sigma_{L}^{F_{g_{2}}}(\eta))|_{*}
=-2\xi_{1}\xi_{2}+\frac{2f\xi_{1}\xi_{2}+2h}{f+2h\xi_{1}\xi_{2}}
-\frac{2[fh+(f^{2}+h^{2})\xi_{1}\xi_{2}]^{2}}{(f+2h\xi_{1}\xi_{2})^{2}};
\end{equation}
\begin{equation}
B_{4}=\partial_{\xi_{2}}\partial_{\eta_{2}}{\rm trace}(\sigma_{L}^{F_{g_{1}}}(\xi) \sigma_{L}^{F_{g_{2}}}(\eta))|_{*}
=1-2\xi_{2}^{2}+\frac{f(1+2\xi_{2}^{2})+2h\xi_{1}\xi_{2}}{f+2h\xi_{1}\xi_{2}}
-\frac{2(f\xi_{2}+h\xi_{1})^{2}}{(f+2h\xi_{1}\xi_{2})^{2}}.
\end{equation}

Since $|\xi|=1$ is the unit circle, let $\xi_{1}=cos\theta, \xi_{2}=sin\theta $. By (4.14) and (4.16), we have

\begin{equation}
\int_{|\xi|=1}B_{1}d\xi
=\int_{0}^{2\pi}(1-2cos^{2}\theta)d\theta
+\int_{0}^{2\pi}\frac{f(1+2cos^{2}\theta)+2h cos\theta sin\theta}{f+2h cos\theta sin\theta}d\theta
-\int_{0}^{2\pi}\frac{2(fcos\theta+h sin\theta)^{2}}{(f+2h cos\theta sin\theta)^{2}}d\theta.
\end{equation}

By the integral formula and $f>h>0$, we get

\begin{equation}
\int_{0}^{2\pi}(1-2cos^{2}\theta)d\theta=0;
~\int_{0}^{2\pi}\frac{1}{sin 2 \theta+\frac{f}{h}} d\theta=\frac{8h}{\sqrt{f^{2}-h^{2}}} arctan\sqrt{\frac{f-h}{f+h}};
\end{equation}
\begin{equation}
\int_{0}^{\frac{\pi}{2}}\frac{1}{(sin\theta+\frac{f}{h})^{2}} d\theta
=\frac{h^{3}}{f(f^{2}-h^{2})}+\frac{2h^{3}}{(f^{2}-h^{2})\sqrt{f^{2}-h^{2}}} arctan\sqrt{\frac{f-h}{f+h}}.
\end{equation}

By (4.20), (4.21) and (4.22), then
\begin{equation}
\int_{|\xi|=1}B_{1}d\xi=2\pi+\frac{4h}{f}-8\sqrt{\frac{f-h}{f+h}} arctan\sqrt{\frac{f-h}{f+h}} .
\end{equation}

Similar to the computation of the first case, by integral formula, then
\begin{eqnarray}
\int_{|\xi|=1}B_{2}d\xi&=&\int_{|\xi|=1}B_{3}d\xi
=\frac{2f\pi}{h}-\frac{28h^{3}+8f^{3}+24h^{2}f+6hf^{2}}{(f+h)(f^{2}-h^{2})}\nonumber\\
&&+\frac{24h^{2}f-8f^{3}-8fh^{2}-8h^{3}}{h(f+h)\sqrt{f^{2}-h^{2}}}arctan\sqrt{\frac{f-h}{f+h}};
\end{eqnarray}

\begin{equation}
\int_{|\xi|=1}B_{4}d\xi=2\pi+\frac{4h}{f}-8\sqrt{\frac{f-h}{f+h}} arctan\sqrt{\frac{f-h}{f+h}}.
\end{equation}

Therefore, for $n=2$, we get the double conformally invariant
\begin{eqnarray}
&&\Omega^{g_{1},g_{2}}_{d}(f_1,f_2)\nonumber\\
&=&D_{x}^{(1,0)}(f_{1}) D_{x}^{(1,0)}(f_{2}) (2\pi+\frac{4h}{f}-8\sqrt{\frac{f-h}{f+h}} arctan\sqrt{\frac{f-h}{f+h}})dvol_{g^{S^{1}}\oplus g^{S^{1}}}\nonumber\\
&&+D_{x}^{(1,0)}(f_{1}) D_{x}^{(0,1)}(f_{2})(\frac{2f\pi}{h}-\frac{28h^{3}+8f^{3}+24h^{2}f+6hf^{2}}{(f+h)(f^{2}-h^{2})}\nonumber\\
&&+\frac{24h^{2}f-8f^{3}-8fh^{2}-8h^{3}}{h(f+h)\sqrt{f^{2}-h^{2}}}arctan\sqrt{\frac{f-h}{f+h}})dvol_{g^{S^{1}}\oplus g^{S^{1}}}\nonumber\\
&&+D_{x}^{(1,0)}(f_{1}) D_{x}^{(0,1)}(f_{2})(\frac{2f\pi}{h}-\frac{28h^{3}+8f^{3}+24h^{2}f+6hf^{2}}{(f+h)(f^{2}-h^{2})}\nonumber\\
&&+\frac{24h^{2}f-8f^{3}-8fh^{2}-8h^{3}}{h(f+h)\sqrt{f^{2}-h^{2}}}arctan\sqrt{\frac{f-h}{f+h}})dvol_{g^{S^{1}}\oplus g^{S^{1}}} \nonumber\\
&&+D_{x}^{(0,1)}(f_{1}) D_{x}^{(0,1)}(f_{2})(2\pi+\frac{4h}{f}-8\sqrt{\frac{f-h}{f+h}} arctan\sqrt{\frac{f-h}{f+h}})dvol_{g^{S^{1}}\oplus g^{S^{1}}}.
\end{eqnarray}

\section{A Double Conformally Invariant Differential Form $\Omega^{g_{1},g_{2}}_{\overline{\partial}}(f_1,f_2)$ for Complex Manifolds}
 \label{5}
 First we recall some basic facts on complex geometry (for details, see  \cite{CCL} or \cite{GH}).
Let $M$ be a compact connect complex manifold with complex dimension $n$. Let $g^{\bf R}$ be
a Riemannian metric on the real tangent bundle $T^{\bf R}M$ of $M$. We canonically extend $g^{{\bf R}}$ to a
Hermitian metric $g^{{\bf C}}$ on the complexified bundle $T^{\bf R}M\otimes {\bf C}$. Denote
by $g^{{1,0}}$ (res. $g^{{0,1}}$ ) the restriction on the holomorphic bundle $T^{1,0}M$ (res. $T^{0,1}M$ )
of  $g^{{\bf C}}$. Let $g^{{\bf R},*}$ (res. $g^{1,0,*}$, $g^{0,1,*}$) be the dual metric on the
dual bundle $T^{\bf R,*}M$ (rep. $T^{1,0,*}M$,$T^{0,1,*}M$). let $A^{p,q}$ denote the smooth
$(p,q)$-forms space on $M$.  Let $J$ be the canonical almost complex structure on $T^{\bf R}M$
and $\{ e_1,\cdots, e_n,e_{n+1},\cdots,e_{2n}\}$ be the local orthonormal basis which satisfies
$Je_i=e_{n+i}$. Thus $g^{\bf R}$ is $J$-invariant metric and

\begin{equation}
\phi_j=\frac{1}{\sqrt{2}}(e_j-\sqrt{-1}e_{n+j});~\overline{\phi_j}=\frac{1}{\sqrt{2}}(e_j+\sqrt{-1}e_{n+j}),
~1\leq j\leq n
\end{equation}
are respectively the orthonormal basis of $T^{1,0}M$ and $T^{0,1}M$. Let $\{ e^1,\cdots,e^n,e^{n+1},\cdots,e^{2n}\}$
be the dual basis on $T^{{\bf R},*}M$ which satisfies $Je^{j}=-e^{n+j}$ where $J$ is the induced almost
complex structure on $T^{{\bf R},*}M$. Then
\begin{equation}
\lambda_j=\frac{1}{\sqrt{2}}(e^j+\sqrt{-1}e^{n+j});~\overline{\lambda_j}=\frac{1}{\sqrt{2}}(e^j-\sqrt{-1}
e^{n+j}),~1\leq j\leq n
\end{equation}
are respectively the orthonormal basis of $T^{1,0,*}M$ and $T^{0,1,*}M$. $\{\phi_j,\overline{\phi_j}\}$ and
$\{\lambda_j,\overline{\lambda_j}\}$ are dual basis. Define the inner product on the smooth $(p,q)$-forms
space $A^{p,q}$ such that the basis $\{\lambda_{i_1}\wedge\cdots,\lambda_{i_p}\wedge
\overline{\lambda_{j_1}}\wedge\cdots\wedge\overline{\lambda_{j_q}}\}$ is the orthonormal basis.
$\Phi=\lambda_1\wedge\cdots\wedge\lambda_n\wedge\overline{\lambda_1}\wedge\cdots\wedge\overline{\lambda_n}$
be the canonical volume form on $M$ up to some constant. Define an inner product for $\psi, \eta\in A^{p,q}$

\begin{equation}
\langle\langle\psi,\eta\rangle\rangle=\int_M\langle\psi,\eta\rangle\Phi,
\end{equation}
which makes $A_{p,q}$ be a preHilbert space and let $L^2(A^{p,q})$ be $L^2$-completion. Recall the
Hodge $\star$ operator $\star:A^{p,q}\rightarrow A^{n-p,n-q}$ defined as follows:

\begin{equation}
\langle\psi,\eta\rangle\Phi=\psi\wedge\star\eta,~ {\rm for}~
\psi,\eta\in A_{p,q}.
\end{equation}
Let $\Delta$ be the
$\overline{\partial}$-Laplacian and $\overline{\partial}^*=\widehat{c}\star\partial\star$ be the adjoint
operator of $\partial$ where $\widehat{c}$ is a constant. Let $H^{p,q}$ be the harmonic $(p,q)$-forms space.
Recall the Hodge Theorem says that (see \cite{GH})

\begin{equation}
A_{p,q}=H^{p,q}\oplus\overline{\partial}A^{p,q-1}\oplus\overline{\partial}^*A^{p,q+1}.
\end{equation}
Now we take the conformal rescaling metric $\widetilde{g}=e^{2f}g^{{\bf R}}$ for some smooth function $f$ on
$M$, then the associated dual metric is $e^{-2f}g^{{\bf R},\star}$ and the associated orthonormal basis are
$e^{-f}\phi_i, e^{-f}\overline{\phi_i}$ and  $e^{f}\lambda_i, e^{f}\overline{\lambda_i}$. Let $\widehat{\star}$ be
the Hodge star operator associated to the rescaling metric.

\begin{lem}\label{le:51}
 $ \widehat{\star}=e^{2(n-p-q)f}\star.$
In particular, when $p+q=n$, the $\star$ operator is conformal invariant, that is $\widehat{\star}=\star.$
\begin{proof}
By equalities $\widehat{\Phi}=e^{2nf}\Phi$ and
$\widetilde{g}(\psi,\eta)=e^{-2(p+q)f}g(\psi,\eta)$ for
$\psi,~\eta\in A^{p,q},$ then we have
\begin{equation}
\psi\wedge\widehat{\star}\eta=\widetilde{g}(\psi,\eta)\widetilde{\Phi}=e^{2(n-p-q)f}g(\psi,\eta){\Phi}
=e^{2(n-p-q)f}\psi\wedge{\star}\eta.
\end{equation}
by (5.6) and the Poincar$\acute{e}$ duality, we get $\widehat{\star}=e^{2(n-p-q)f}\star$.
\end{proof}
\end{lem}

  By the definition of the inner product, when $n=p+q$, $L^2(A^{p,q})$ is conformally invariant. In the following, we
 always assume $n=p+q$. By the Hodge decomposition theorem (also see \cite{Ug2}), we have
 $H^{p,q}={\rm ker}(\overline{\partial})\cap {\rm ker}(\overline{\partial}\star).$ So by Lemma 5.1 and
 $\overline{\partial}$ is independent of metric, we get that $H^{p,q}$ is conformally invariant.
 Let $F_{\overline{\partial}}=2P_{\rm Im\overline{\partial}}-{\rm Id}$ where $P_{\rm Im\overline{\partial}}$
 is the projection from $L^2(A^{p,q})$ to $ \rm Im\overline{\partial}$. As in [Co], we have on $A^{p,q}\ominus H^{p,q}$

\begin{equation}
F_{\overline{\partial}}=\frac{{\bar{\partial}} \bar{\partial}^*-\overline{\partial}^*\overline{\partial}}
{\bar{\partial}\bar{\partial}^*+\overline{\partial}^*\overline{\partial}}.
\end{equation}

Thus $F_{\overline{\partial}}$ is a conformally invariant $0$-order pseudodifferential
bounded operator satisfying $F_{\overline{\partial}}=F_{\overline{\partial}}^*$ and
$F_{\overline{\partial}}^2=1$. Let $C^{\infty}(M)$ be the set of smooth complex-valued functions on $M$.
Then we have

\begin{lem}\label{le:52}
If $M$ is a compact conformal complex manifold without boundary, then $(C^{\infty}( M), L^2(A^{p,q}),
F_{\overline{\partial}})$ is a conformally invariant Fredholm module (for definition, see \cite{GVF}) .
\end{lem}

Let $g_{1}^{{\bf R},TM},~ g_{2}^{{\bf R},TM}$ be the two different metrics on $T^{{\bf R},*}M$.
Let $F^{g_{i}}_{\overline{\partial}}=2P_{\rm Im\overline{\partial}}-{\rm Id}~(i=1, 2)$ where $P_{\rm Im\overline{\partial}}$
 is the projection from $L^2(A^{p,q})$ to $ \rm Im\overline{\partial}$. As in \cite{Co}, we have on $A^{p,q}\ominus H^{p,q}$,
  for $i=1,~2$

\begin{equation}
F^{g_{i}}_{\overline{\partial}}=\frac{{\bar{\partial}} \bar{\partial}_{g_{i}}^*-\overline{\partial}_{g_{i}}^*\overline{\partial}}
{\bar{\partial}\bar{\partial}_{g_{i}}^*+\overline{\partial}_{g_{i}}^*\overline{\partial}},
\end{equation}

Thus $F^{g_{i}}_{\overline{\partial}} ~(i=1,2)$ is a conformally invariant $0$-order pseudodifferential bounded operator.
For $n-$dimensional complex manifold,  we define the $2n$-form
$\Omega^{g_{1},g_{2}}_{\overline{\partial}}(f_1,f_2)$ by the following identity for $f_0, f_1, f_2\in C^{\infty}( M)$

\begin{equation}
Wres(f_0[F^{g_{1}}_{\overline{\partial}},f_1][F^{g_{2}}_{\overline{\partial}},f_2])
=\int_Mf_0\Omega^{g_{1},g_{1}}_{\overline{\partial}}(f_1,f_2).
\end{equation}

By Lemma 5.2, as in \cite{Co} or \cite{Ug1}, we have
\begin{thm}\label{th:53}
$\Omega^{g_{1},g_{2}}_{\overline{\partial}}(f_1,f_2)$ is a unique defined by (5.9), symmetric and conformally invariant
$2n$-differential form. Furthermore, $ \int_Mf_0\Omega^{g_{1},g_{2}}_{\overline{\partial}}(f_1,f_2)$ defines a
Hochschild $2$-cocycle on the algebra of smooth functions on $M$.
\end{thm}

\section{The Computation of the Conformal Invariant $\Omega^{g_{1},g_{2}}_{\overline{\partial}}(f_1,f_2)$}
\label{6}
Next we compute the conformal invariant $\Omega^{g_{1},g_{2}}_{\overline{\partial}}(f_1,f_2)$. By Lemmas 2.2 and 2.3
in \cite{Ug1}, we have

\begin{equation}
\Omega^{g_{1},g_{2}}_{\overline{\partial}}(f_1,f_2)=\int_{|\xi|=1}{\rm tr}_{\wedge^{p,q}}
[\sum\frac{1}{\alpha'!\alpha''!\beta!\delta!}D^{\beta}_x({f_1}) D^{\alpha''+\delta}_x({f_2})
\partial^{\alpha'+\alpha''+\beta}_{\xi}(\sigma^{F^{g_{1}}_{\overline{\partial}}}_{-j})
\partial^{\delta}_{\xi}D^{\alpha'}_x(\sigma^{F^{g_{2}}_{\overline{\partial}}}_{-k})]\sigma(\xi)d^nx ,
\end{equation}
where $\wedge^{p,q}$ is the space of $(p,q)$-degree forms and $\sigma_{-j}^{F^{g_{1}}_{\overline{\partial}}},\sigma_{-j}^{F^{g_{2}}_{\overline{\partial}}}$
denotes the order $-j$ symbol of $F^{g_{1}}_{\overline{\partial}}, F^{g_{2}}_{\overline{\partial}}$;
$D^\beta_x=(-i)^{|\beta|}\partial^\beta_x$ and the sum is taken over $|\alpha'|+|\alpha''|+|\beta|+|\delta|+j+k=2n;
|\beta|\geq1,|\delta|\geq1; \alpha',\alpha'',\beta,\delta \in {\bf Z}^{2n}_+; j,k\in {\bf Z}_+.$
Recall Lemma 3.52 in \cite{Gi}, then we have

\begin{lem}\label{le:61}
The following equalities hold
\begin{equation}
\sigma_L(\overline{\partial})(x,\xi)=\frac{\sqrt{-1}}{2}\sum_{1\leq j\leq n}(\xi_j+\sqrt{-1}\xi_{j+n})\varepsilon(e^j-\sqrt{-1}e^{j+n});
\end{equation}

\begin{equation}
\sigma_L(\overline{\partial}^*)(x,\xi)=-\frac{\sqrt{-1}}{2}\sum_{1\leq j\leq n}(\xi_j-\sqrt{-1}\xi_{j+n})\iota(e^j-\sqrt{-1}e^{j+n});~
\sigma_L(\triangle)=\frac{1}{2}|\xi|^2{\rm Id}.
\end{equation}
\end{lem}

\begin{lem}\label{le:62}
When $g_{i}~(i=1,2)$ is flat, we have for $\xi\neq 0$
\begin{equation}
\sigma(\overline{\partial})=\sigma_L(\overline{\partial}); \sigma(\overline{\partial}^*)=\sigma_L(\overline{\partial}^*);
~\sigma(F^{g_{i}}_{\overline{\partial}})=\sigma_L(F^{g_{i}}_{\overline{\partial}})=2|\xi|^{-2}
[\sigma^{g_{i}}_L(\overline{\partial})\sigma^{g_{i}}_L(\overline{\partial}^*)
-\sigma^{g_{i}}_L(\overline{\partial}^*)\sigma^{g_{i}}_L(\overline{\partial})].
\end{equation}
\end{lem}
We assume that $g_{i}~(i=1,2)$ is flat. Similarly to Lemma 6.1 in \cite{Ug1}, then we have

\begin{lem}\label{le:63}
The following equality holds
\begin{equation}
\Omega^{g_{1},g_{2}}_{\overline{\partial},{\rm flat}}(f_1,f_2)=\int_{|\xi|=1}{\rm tr}_{\wedge^{p,q}}
[\sum\frac{1}{\alpha!\beta!\delta!}D^{\beta}_x({f_1})D^{\alpha+\delta}_x({f_2})
\partial^{\alpha+\beta}_{\xi}(\sigma_L(F^{g_{1}}_{\overline{\partial}}))
\partial^{\delta}_{\xi}(\sigma_L(F^{g_{2}}_{\overline{\partial}}))]\sigma(\xi)d^nx,
\end{equation}
where the sum is taken over $|\alpha|+|\beta|+|\delta|=2n; ~1\leq |\beta|; ~1\leq |\delta|.$
\end{lem}
Let
\begin{equation}
\psi(\xi,\eta)={\rm Tr}_{\wedge^{p,q}}[\sigma_L(F^{g_{1}}_{\overline{\partial}})(x,\xi)\sigma_L(F^{g_{2}}_{\overline{\partial}})(x,\eta)];
\end{equation}

\begin{equation}
T_{2n}'\psi(\xi,\eta,u,v)=\sum\frac{u^{\beta}}{\beta!}\frac{v^{\delta}}{\delta!}{\rm tr}[\partial^\beta_\xi
(\sigma_L(F^{g_{1}}_{\overline{\partial}})(\xi))\partial^\delta_\eta(\sigma_L(F^{g_{2}}_{\overline{\partial}})(\eta))],
\end{equation}
where the sum is taken over $|\beta|+|\delta|=2n,~|\beta|\geq 1,~|\delta|\geq 1$. $T_{2n}'(\xi,\eta,u,v)$ is the
term of order $2n$ in the Taylor expansion of $\psi(\xi+u,\eta+v)$ at $u=v=0$ minus the terms with only powers
 of $u$ or only powers of $v$. As in Section 4 in \cite{Ug1}, we have

\begin{thm}\label{th:64}
\begin{equation}
\Omega^{g_{1},g_{2}}_{\overline{\partial},{\rm flat}}(f_1,f_2)=(\sum A_{a, b}(D^a_xf_1)(D^b_xf_2))dx,
\end{equation}
where
\begin{equation}
\sum A_{a, b}u^av^b=\int_{|\xi|=1}(T_{2n}'\psi(\xi,\xi,u+v,v)-T_{2n}'\psi(\xi,\xi,v,v))\sigma(\xi).
\end{equation}
\end{thm}
By Theorem 6.4, to obtain an explicit expression for $\Omega^{g_{1},g_{2}}_{\overline{\partial},{\rm flat}}$, it is
necessary to study ${\rm Tr}_{\wedge^{p,q}}[\sigma_L(F^{g_{1}}_{\overline{\partial}})(\xi)$ $\sigma_L(F^{g_{2}}_{\overline{\partial}})(\eta)]$
for $\xi$ and $\eta$ not zero in $T^{R,*}_xM$. For $\xi=\sum_{1\leq i \leq 2n}\xi_ie^i$ , we let

\begin{equation}
\widehat{\xi}=\sum_{1\leq j\leq n}(\xi_j+\sqrt{-1}\xi_{j+n})(e^j-\sqrt{-1}e^{j+n}).
\end{equation}

For a $(0,1)$-form $\omega_1$, we have

\begin{eqnarray}
\iota^{g_{i}}((\xi_j+\sqrt{-1}\xi_{j+n})(e^j-\sqrt{-1}e^{j+n}))\omega_1&=&\langle\omega_1,(\xi_j+\sqrt{-1}\xi_{j+n})(e^j-\sqrt{-1}e^{j+n})\rangle \nonumber\\
                                                               &=&(\xi_j-\sqrt{-1}\xi_{j+n})\iota^{g_{i}}(e^j-\sqrt{-1}e^{j+n})\omega_1.
\end{eqnarray}

By Lemma 6.1 and (6.11), we have
\begin{equation}
\sigma_L^{g_{i}}(\overline{\partial})(\xi)=\frac{\sqrt{-1}}{2}\varepsilon(\widehat{\xi});~
\sigma_L^{g_{i}}(\overline{\partial}^*)(\xi)=-\frac{\sqrt{-1}}{2}\iota^{g_{i}}(\widehat{\xi}).
\end{equation}

\begin{thm}\label{th:65}
For $q=1, p+1=n$
\begin{equation}
{\rm Tr}_{\wedge^{p,1}}[\sigma_L(F^{g_{1}}_{\overline{\partial}})(\xi)\sigma_L(F^{g_{2}}_{\overline{\partial}})(\eta)]
=C^p_n\frac{\overline{\langle\widehat{\xi},\widehat{\eta}\rangle}_{g_{1}}\langle\widehat{\xi},\widehat{\eta}\rangle_{g_{2}}}{|\xi|_{g_{1}}^2|\eta|_{g_{2}}^2}
+C_n^pC_n^1-4C^p_n.
\end{equation}

For $q=2, p+2=n$
\begin{eqnarray}
&&{\rm Tr}_{\wedge^{p,2}}[\sigma_L(F^{g_{1}}_{\overline{\partial}})(\xi)\sigma_L(F^{g_{2}}_{\overline{\partial}})(\eta)]\nonumber\\
&=&\frac{\overline{\langle\widehat{\xi},\widehat{\eta}\rangle}_{g_{1}}\langle\widehat{\xi},\widehat{\eta}\rangle_{g_{2}}(2C_n^pA_{n,1}-C_n^pC_n^1)}
{|\xi|_{g_{1}}^2|\eta|_{g_{2}}^2}+4C_n^p+C_n^pC_n^2-4C_n^pA_{n,1}.
\end{eqnarray}

For $q\geq 3, p+q=n$
\begin{equation}
{\rm Tr}_{\wedge^{p,q}}[\sigma_L(F^{g_{1}}_{\overline{\partial}})(\xi)\sigma_L(F^{g_{2}}_{\overline{\partial}})(\eta)]
=a_{p,q}\frac{\overline{\langle\widehat{\xi},\widehat{\eta}\rangle}_{g_{1}}\langle\widehat{\xi},\widehat{\eta}\rangle_{g_{2}}}{|\xi|_{g_{1}}^2|\eta|_{g_{2}}^2}+b_{p,q},
\end{equation}
where $\langle\widehat{\xi},\widehat{\eta}\rangle_{g_{i}}(i=1,2)$ represents the inner product $g^{{\bf C},*}(\widehat{\xi},\widehat{\eta})$,
$a_{p,q}$ and $b_{p,q}$ are some constants, $C_n^p$ is a combinatorial number.
\end{thm}

\begin{lem}\label{le:66}
For $\widetilde{\xi},\widetilde{\eta}\in A_{0,1},$
\begin{equation}
\rm tr(\iota^{g_{1}}_{p,q+1}(\widetilde{\eta})\varepsilon_{p,q}(\widetilde{\xi}))=\langle\widetilde{\xi},\widetilde{\eta}\rangle_{g_{1}}C_n^pA_{n,q},
\end{equation}
where $A_{n,q}=C_n^q-C_n^{q-1}+\cdots +(-1)^q C_n^0$ and $\varepsilon_{p,q} $ denotes the operator $ \varepsilon $ acting
on the space of $(p,q)$ forms.
\end{lem}

\begin{proof}
By using the {\rm trace} property and the relation
\begin{equation}
\varepsilon_{p,q-1}(\widetilde{\xi})\iota^{g_{1}}_{p,q}(\widetilde{\eta})+\iota^{g_{1}}_{p,q+1}(\widetilde{\eta})\varepsilon_{p,q}
(\widetilde{\xi})=\langle\widetilde{\xi},\widetilde{\eta}\rangle_{g_{1}}I_{p,q},
\end{equation}
where $I_{p,q}$ is the identity on $(p,q)$-forms, we get
\begin{equation}
{\rm tr}[\iota^{g_{1}}_{p,q}(\widetilde{\eta})\varepsilon_{p,q-1}(\widetilde{\xi})]+{\rm tr}[\iota^{g_{1}}_{p,q+1}(\widetilde{\eta})
\varepsilon_{p,q}(\widetilde{\xi})]=\langle\widetilde{\xi},\widetilde{\eta}\rangle_{g_{1}}C_n^pC_n^q.
\end{equation}
When $q=0$, we have
${\rm tr}[\iota^{g_{1}}_{p,1}(\widetilde{\eta})\varepsilon^{g_{1}}_{p,0}(\widetilde{\xi})]=\langle\widetilde{\xi},\widetilde{\eta}\rangle_{g_{1}}C_n^p$.
Then by (6.18), we can prove this Lemma.
\end{proof}

\noindent{\bf Proof of Theorem 6.5}
\begin{proof}
By Lemma 6.1 and (6.12), we have
\begin{eqnarray}
&&{\rm Tr}_{\wedge^{p,q}}[\sigma_L(F^{g_{1}}_{\overline{\partial}})(\xi)\sigma_L(F^{g_{2}}_{\overline{\partial}})(\eta)]\nonumber\\
&=&{\rm Tr}_{\wedge^{p,q}}\{2|\xi|^{-2}[2\sigma^{g_{1}}_L(\overline{\partial})\sigma^{g_{1}}_L(\overline{\partial}^*)-\sigma^{g_{1}}_L(\triangle)]
(\xi)2|\eta|^{-2}[2\sigma^{g_{2}}_L(\overline{\partial})\sigma^{g_{2}}_L(\overline{\partial}^*)-\sigma_L^{g_{2}}(\triangle)](\eta)\} \nonumber\\
&=&|\xi|_{g_{1}}^{-2}|\eta|_{g_{2}}^{-2}{\rm tr}[\varepsilon_{p,q-1}(\widehat{\xi})\iota^{g_{1}}_{p,q}(\widehat{\xi})
\varepsilon_{p,q-1}(\widehat{\eta})\iota^{g_{2}}_{p,q}(\widehat{\eta})] -|\xi|_{g_{1}}^{-2}{\rm tr}
[\varepsilon_{p,q-1}(\widehat{\xi})\iota^{g_{1}}_{p,q}(\widehat{\xi})] \nonumber\\
&&-|\eta|_{g_{2}}^{-2}{\rm tr}[\varepsilon_{p,q-1}(\widehat{\eta})\iota^{g_{2}}_{p,q}(\widehat{\eta})]+C_n^pC^q_n.
\end{eqnarray}

By Lemma 6.6, we have
\begin{eqnarray}
&&|\xi|_{g_{1}}^{-2}{\rm tr}[\varepsilon_{p,q-1}(\widehat{\xi})\iota^{g_{1}}_{p,q}(\widehat{\xi})]
=|\xi|_{g_{1}}^{-2}{\rm tr}[\iota^{g_{1}}_{p,q}(\widehat{\xi})\varepsilon_{p,q-1}(\widehat{\xi})]  \nonumber\\
&=&|\xi|_{g_{1}}^{-2}\langle\widehat{\xi},\widehat{\xi}\rangle_{g_{1}}C_n^pA_{n,q-1}=|\xi|_{g_{1}}^{-2}2|\xi|_{g_{1}}^{2}C_n^p A_{n,q-1}=2C_n^pA_{n,q-1}.
\end{eqnarray}

For $q\geq 1$ and $\xi_a,\xi_b,\eta_a,\eta_b\in A_{0,1}$, let

\begin{equation}
B_{p,q}(\xi_a,\xi_b,\eta_a,\eta_b)={\rm tr}[\varepsilon_{p,q-1}(\xi_a)\iota^{g_{1}}_{p,q}(\xi_b)
\varepsilon_{p,q-1}(\eta_a)\iota^{g_{2}}_{p,q}(\eta_b)].
\end{equation}

 By (6.17) and Lemma 6.6 and the {\rm trace} property, we have the relation
\begin{eqnarray}
&&B_{p,q}(\xi_a,\xi_b,\eta_a,\eta_b) \nonumber\\
&=&{\rm tr}\{[<\xi_a,\xi_b>_{g_{1}}I_{p,q}-\iota^{g_{1}}_{p,q+1}(\xi_b)\varepsilon_{p,q}(\xi_a)]
[\langle\eta_a,\eta_b\rangle_{g_{2}}I_{p,q}-\iota^{g_{2}}_{p,q+1}(\eta_b)\varepsilon_{p,q}(\eta_a)]\} \nonumber\\
&=&\langle\xi_a,\xi_b\rangle_{g_{1}}\langle\eta_a,\eta_b\rangle_{g_{2}}(C_n^pC_n^q-2C_n^pA_{n,q})+{\rm tr}[\varepsilon_{p,q}(\eta_a)\iota^{g_{1}}_{p,q+1}(\xi_b)
\varepsilon_{p,q}(\xi_a)\iota^{g_{2}}_{p,q+1}(\eta_b)],
\end{eqnarray}
which implies

\begin{equation}
B_{p,q+1}(\eta_a,\xi_b,\xi_a,\eta_b)=B_{p,q}(\xi_a,\xi_b,\eta_a,\eta_b)+\langle\xi_a,\xi_b\rangle_{g_{1}}\langle\eta_a,\eta_b\rangle_{g_{2}}(2C_n^pA_{n,q}-C_n^pC_n^q),
\end{equation}
\begin{eqnarray}
B_{p,1}(\xi_a,\xi_b,\eta_a,\eta_b)&=&{\rm tr}[\varepsilon_{p,0}(\xi_a)\iota^{g_{1}}_{p,1}(\xi_b)
\varepsilon_{p,0}(\eta_a)\iota^{g_{2}}_{p,1}(\eta_b)]\nonumber\\
&=&{\rm tr}[\iota^{g_{1}}_{p,1}(\eta_b)\varepsilon_{p,0}(\xi_a)\iota^{g_{2}}_{p,1}(\xi_b)\varepsilon_{p,0}(\eta_a)]
=\langle\eta_a,\xi_b\rangle_{g_{1}}\langle\xi_a,\eta_b\rangle_{g_{2}}C_n^p.
\end{eqnarray}

By (6.19) and (6.20), we have
\begin{equation}
{\rm Tr}_{\wedge^{p,q}}[\sigma_L(F^{g_{1}}_{\overline{\partial}})(\xi)\sigma_L(F^{g_{2}}_{\overline{\partial}})(\eta)]
=|\xi|_{g_{1}}^{-2}|\eta|_{g_{2}}^{-2}B_{p,q}(\widehat{\xi},\widehat{\xi},\widehat{\eta},\widehat{\eta})+C_n^pC_n^q-4C_n^pA_{n,q-1},
\end{equation}
\begin{equation}
B_{p,1}(\widehat{\xi},\widehat{\xi},\widehat{\eta},\widehat{\eta})
=\overline{\langle\widehat{\xi},\widehat{\eta}\rangle}_{g_{1}}\langle\widehat{\xi},\widehat{\eta}\rangle_{g_{2}}C_n^p,
\end{equation}
and
\begin{eqnarray}
B_{p,2}(\widehat{\xi},\widehat{\xi},\widehat{\eta},\widehat{\eta})
&=&B_{p,1}(\widehat{\eta},\widehat{\xi},\widehat{\xi},\widehat{\eta})
+\overline{\langle\widehat{\xi},\widehat{\eta}\rangle}_{g_{1}}\langle\widehat{\xi},\widehat{\eta}\rangle_{g_{2}}(2C_n^pA_{n,1}-C_n^pC_n^1)   \nonumber\\
&=&\langle\widehat{\xi},\widehat{\xi}\rangle_{g_{1}}\langle\widehat{\eta},\widehat{\eta}\rangle_{g_{2}}C_n^p
+\overline{\langle\widehat{\xi},\widehat{\eta}\rangle}_{g_{1}}\langle\widehat{\xi},\widehat{\eta}\rangle_{g_{2}}(2C_n^pA_{n,1}-C_n^pC_n^1).
\end{eqnarray}

By (6.25) and (6.26), we can get (6.13). By (6.25) and (6.27), we can get (6.14). By (6.25) and a recursive way, we have
\begin{equation}
{\rm Tr}_{\wedge^{p,q}}[\sigma_L(F^{g_{1}}_{\overline{\partial}})(\xi)\sigma_L(F^{g_{2}}_{\overline{\partial}})(\eta)]
=a_{p,q}\frac{\overline{\langle\widehat{\xi},\widehat{\eta}\rangle}_{g_{1}}\langle\widehat{\xi},\widehat{\eta}\rangle_{g_{2}}}{|\xi|_{g_{1}}^2|\eta|_{g_{2}}^2}+b_{p,q}.
\end{equation}
\end{proof}

\section{ $\Omega^{g_{1},g_{2}}_{\overline{\partial}}(f_1,f_2)$ for 1-dimensional Complex Manifolds}
\label{7}
In this section, we compute the double conformal invariant for 1-dimension complex manifolds.
By lemma 6.3 (for the definition of $\Omega^{g_{1},g_{2}}_{\overline{\partial}}(f_1,f_2)$), since the sum is taken
$|\alpha|+|\beta|+|\delta|=2; ~1\leq |\beta|; ~1\leq |\delta|$, we get $|\alpha|=0,~|\beta|=|\delta|=1$, then we have the four cases:
$\{\beta =(1,0),~ \delta=(1,0)\};~\{\beta =(1,0),~\delta =(0,1)\};~\{\beta =(0,1),~\delta =(1,0)\};~\{\beta =(0,1),~\delta =(0,1)\}.$

Then we have
\begin{eqnarray}
\Omega^{g_{1},g_{2}}_{\overline{\partial}}(f_1,f_2)
&=&\int_{|\xi|=1}{\rm tr}_{\wedge^{0,1}}
[\sum\frac{1}{\alpha!\beta!\delta!}D^{\beta}_x({f_1})D^{\alpha+\delta}_x({f_2})\partial^{\alpha+\beta}_{\xi}(\sigma_L(F^{g_{2}}_{\overline{\partial}}))
\partial^{\delta}_{\xi}(\sigma_L(F^{g_{2}}_{\overline{\partial}}))]\sigma(\xi)dx \nonumber\\
&=&\int_{|\xi|=1}\frac{1}{1!0!1!0!} D_{x}^{(1,0)}(f_{1}) D_{x}^{(1,0)}(f_{2})\partial_{(\xi_{1},~\xi_{2})}^{(1,0)}\partial_{(\eta_{1},~
\eta_{2})}^{(1,0)} {\rm trace}(\sigma_{L}^{F_{g_{1}}}(\xi) \sigma_{L}^{F_{g_{2}}}(\eta)) d\xi dx \nonumber\\
&&+\int_{|\xi|=1}\frac{1}{1!0!0!1!} D_{x}^{(1,0)}(f_{1}) D_{x}^{(0,1)}(f_{2})\partial_{(\xi_{1},~\xi_{2})}^{(1,0)}\partial_{(\eta_{1},~
\eta_{2})}^{(0,1)} {\rm trace}(\sigma_{L}^{F_{g_{1}}}(\xi) \sigma_{L}^{F_{g_{2}}}(\eta)) d\xi dx \nonumber\\
&&+\int_{|\xi|=1}\frac{1}{0!1!1!0!} D_{x}^{(0,1)}(f_{1}) D_{x}^{(1,0)}(f_{2})\partial_{(\xi_{1},~\xi_{2})}^{(0,1)}\partial_{(\eta_{1},~
\eta_{2})}^{(1,0)} {\rm trace}(\sigma_{L}^{F_{g_{1}}}(\xi) \sigma_{L}^{F_{g_{2}}}(\eta)) d\xi dx \nonumber\\
&&+\int_{|\xi|=1}\frac{1}{0!1!0!1!} D_{x}^{(0,1)}(f_{1}) D_{x}^{(0,1)}(f_{2})\partial_{(\xi_{1},~\xi_{2})}^{(0,1)}\partial_{(\eta_{1},~
\eta_{2})}^{(0,1)} {\rm trace}(\sigma_{L}^{F_{g_{1}}}(\xi) \sigma_{L}^{F_{g_{2}}}(\eta)) d\xi dx \nonumber\\
&=&D_{x}^{(1,0)}(f_{1}) D_{x}^{(1,0)}(f_{2})\int_{|\xi|=1}\partial_{\xi_{1}}\partial_{\eta_{1}}
 {\rm trace}(\sigma_{L}^{F_{g_{1}}}(\xi) \sigma_{L}^{F_{g_{2}}}(\eta)) d\xi dx \nonumber\\
&&+D_{x}^{(1,0)}(f_{1}) D_{x}^{(0,1)}(f_{2})\int_{|\xi|=1}\partial_{\xi_{1}}\partial_{
\eta_{2}} {\rm trace}(\sigma_{L}^{F_{g_{1}}}(\xi) \sigma_{L}^{F_{g_{2}}}(\eta)) d\xi dx \nonumber\\
&&+D_{x}^{(0,1)}(f_{1}) D_{x}^{(1,0)}(f_{2})\int_{|\xi|=1}\partial_{\xi_{2}}\partial_{\eta_{1}
} {\rm trace}(\sigma_{L}^{F_{g_{1}}}(\xi) \sigma_{L}^{F_{g_{2}}}(\eta)) d\xi dx \nonumber\\
&&+D_{x}^{(0,1)}(f_{1}) D_{x}^{(0,1)}(f_{2})\int_{|\xi|=1}\partial_{\xi_{2}}\partial_{
\eta_{2}}{\rm trace}(\sigma_{L}^{F_{g_{1}}}(\xi) \sigma_{L}^{F_{g_{2}}}(\eta)) d\xi dx.
\end{eqnarray}

For $\xi=\sum_{1\leq i \leq 2n}\xi_ie^i$, let $g_{1}$ be a Riemannian metric on the real tangent bundle $T^{\bf R}M$ of $M$,
and define $g_{2}(e_{i},e_{j})=\langle e_{i},e_{j}\rangle_{g_{2}}$. By (6.10) and direct computation, we get
\begin{equation}
\overline{\langle\widehat{\xi},\widehat{\eta}\rangle}_{g_{1}}=2(\xi_{1}-i\xi_{2})(\eta_{1}+i\eta_{2}),
\end{equation}
\begin{equation}
\langle\widehat{\xi},\widehat{\eta}\rangle_{g_{2}}=(\xi_{1}+i\xi_{2})(\eta_{1}-i\eta_{2})(\langle e_{1},e_{1}\rangle_{g_{2}}+\langle e_{2},e_{2}\rangle_{g_{2}}).
\end{equation}

In this subsection we denote $|_{|\xi|_{g_{1}}=1}^{\eta=\xi} $ by $|_{*}$. By (7.2),(7.3) and Theorem 6.5, we have
\begin{equation}
{\rm Tr}_{\wedge^{0,1}}[\sigma_L(F^{g_{1}}_{\overline{\partial}})(\xi)\sigma_L(F^{g_{2}}_{\overline{\partial}})(\eta)]
=\frac{2(\xi_{1}^{2}+\xi_{2}^{2})(\eta_{1}^{2}+\eta_{2}^{2})(\langle e_{1},e_{1}\rangle_{g_{2}}+\langle e_{2},e_{2}\rangle_{g_{2}})}{|\xi|_{g_{1}}^2|\eta|_{g_{2}}^2}-3.
\end{equation}

By (7.4), we get
\begin{eqnarray}
&&\partial_{\xi_{1}}\partial_{\eta_{1}}{\rm trace}(\sigma_{L}^{F_{g_{1}}}(\xi) \sigma_{L}^{F_{g_{2}}}(\eta))=
(\langle e_{1},e_{1}\rangle_{g_{2}}+\langle e_{2},e_{2}\rangle_{g_{2}}) \nonumber\\
&\times&[\frac{8\xi_{1}\eta_{1}|\xi|_{g_{1}}^2-8\xi_{1}\eta_{1}(\xi_{1}^{2}+\xi_{2}^{2})}{|\xi|_{g_{1}}^4|\eta|_{g_{2}}^2}
-\langle e_{1},\eta\rangle_{g_{2}}\times\frac{8\xi_{1}(\eta_{1}^{2}+\eta_{2}^{2})|\xi|_{g_{1}}^2 -8\xi_{1}(\xi_{1}^{2}+\xi_{2}^{2})
(\eta_{1}^{2}+\eta_{2}^{2})}{|\xi|_{g_{1}}^4|\eta|_{g_{2}}^4}].
\end{eqnarray}

Then
\begin{equation}
\partial_{\xi_{1}}\partial_{\eta_{1}}{\rm trace}(\sigma_{L}^{F_{g_{1}}}(\xi) \sigma_{L}^{F_{g_{2}}}(\eta))|_{*}=0.
\end{equation}

Similar to the computation of the  first case, so we get

\begin{equation}
\partial_{\xi_{1}}\partial_{\eta_{2}}{\rm trace}(\sigma_{L}^{F_{g_{1}}}(\xi) \sigma_{L}^{F_{g_{2}}}(\eta))|_{*}=0;
~~\partial_{\xi_{2}}\partial_{\eta_{1}}{\rm trace}(\sigma_{L}^{F_{g_{1}}}(\xi) \sigma_{L}^{F_{g_{2}}}(\eta))|_{*}=0;
\end{equation}
\begin{equation}
\partial_{\xi_{2}}\partial_{\eta_{2}}{\rm trace}(\sigma_{L}^{F_{g_{1}}}(\xi) \sigma_{L}^{F_{g_{2}}}(\eta))|_{*}=0.
\end{equation}

Then we deduce the conformal invariant for 1-dimension complex manifolds
\begin{thm}\label{th:71}
For 1-dimension complex manifolds, $\Omega^{g_{1},g_{2}}_{\overline{\partial}}(f_1,f_2)=0$.
\end{thm}

By Theorem 6.5 and Theorem 7.1, we have
\begin{con}\label{le:72}
 For any dimensional complex manifolds, $\Omega^{g_{1},g_{2}}_{\overline{\partial}}(f_1,f_2)\equiv 0$.
\end{con}

\begin{rem}\label{th:73}
In this paper, for compact real manifolds, a new double conformal invariant is constructed using the Wodzicki residue and the $d$ operator
in the framework of Connes. In the flat case, we compute this double conformal invariant, and in some special case, we also
compute this double conformal invariants. For complex manifolds, a new double conformal invariant is constructed using
the Wodzicki residue and the $\bar{\partial}$ operator in the same way, and this double conformal invariant is computed
in the flat case. We hope to compute the double conformal invariants in general case in the future.
\end{rem}
\section*{ Acknowledgements}
The second author was partially supported by National Science Foundation of China under Grant No.10801027, and Fok Ying Tong
Education Foundation under Grant No.121003.


\section*{References}
\begin{thebibliography}{00}
\bibitem[CCL]{CCL} Chern.S. S.  Chen W. H. and  Lam. K. S., Lectures on Differential Geometry, Series on University Mathematics, 1. World
             Scientific Publishing Co, Inc, River Edge, NJ, (1999).
\bibitem[Co]{Co}  Connes.A. Quantized Calculus and Applications, XIth International Congress of Mathematical Physics (paris,1994), 15-16,
             Internat Press, Cambridge: MA, (1995).
\bibitem[FGLS]{FGLS} Fedosov. B. v. ,  Golse. F.,  Leichenam. E., Schrohe. E. , The Noncommutative Residue for Manifolds with Boundary,
             J. Funct. Anal. 142: 1-31, (1996).
\bibitem[Gi]{Gi} Gilkey.P. B. , Invariance theory, the Heat Equation and the Atiyah-Singer Index Theorem, Mathematics Lecture Series, 11,
 Publish or Perish, Inc., Wilmington, DE, (1984).
\bibitem[GVF]{GVF} Gracia-Bond\'{i}a. M. , V$\acute{a}$ rilly. J. C.  and Figueroa. H., Elements of Noncommutative Geometry, Birkh\"{a}user Boston,(2001).
\bibitem[GH]{GH} Griffith. P. , and Harris. J. , Priciples of Algebraic Geometry, Wiley, New York, (1978)
\bibitem[Ug1]{Ug1} Ugalde. W. J. , Differential Forms and the Wodzicki Residue, J. Geom. Phys. 58(12), 1739-1751,(2008).
\bibitem[Ug2]{Ug2} Ugalde. W. J. , A Construction of Critical GJMS Operators using Wodzicki's Residue, Comm. Math. Phys, 261(3): 771-788, (2006).
\bibitem[Wa1]{Wa1} Wang. Y., Differential Forms and the Wodzicki Residue for Manifolds with Boundary, J. Germ. Phys, 56(5), 731-753, (2006).
\bibitem[Wa2]{Wa2} Wang. Y., Differential Forms and the Noncommutative Residue for Manifolds with Boundary in the Nonproduct Case,
         Letter. Math. Phys, 77(1), 41-51, (2006).
\bibitem[Wo]{Wo} Wodzicki. M. , Local Invariants of Spectral Asymmetry, Invent. Math, 75, 143-187, (1984).
\bibitem[Zu1]{Zu1} Zucchini. R., A Polyakov Action on Riemann Surfaces, Phys. lett. B, 260(3-4), 296-302, (1991).
\bibitem[Zu2]{Zu2}Zucchini. R., A Polyakov Action on Riemann Surfaces, II. Comm. Math. Phys, 152(2), 269-297, (1993).






\end{thebibliography}
\end{document}